\documentclass[12pt, reqno]{amsart}
\usepackage{amsmath, amsthm, amscd, amsfonts, amssymb, graphicx, color, tikz}
\definecolor{pipeblue}{RGB}{120,180,235}
\usepackage[bookmarksnumbered, colorlinks, plainpages]{hyperref}
\usepackage{comment}

\newtheorem{theorem}{Theorem}[section]

\newtheorem{lemma}[theorem]{Lemma}

\theoremstyle{definition}
\newtheorem{definition}[theorem]{Definition}
\newtheorem{example}[theorem]{Example}

\theoremstyle{remark}

\numberwithin{equation}{section}

\newcommand{\elbowarc}{%
  \begin{tikzpicture}[baseline=-0.5ex, scale=0.3]
    \draw (1.4,0) arc[start angle=90, end angle=180, radius=0.7];
    \draw (0,0) arc[start angle=270, end angle=360, radius=0.7];
  \end{tikzpicture}%
}

\newcommand{\cross}{%
  \begin{tikzpicture}[baseline=0.5ex, scale=0.3]
    \draw (0.7,0)--(0.7,1.4);
    \draw (0,0.7) -- (1.4,0.7);
  \end{tikzpicture}%
}

\begin{document}
\setcounter{page}{1}
\centerline{}

\centerline{}


\title[Schubert polynomials and elementary symmetric products]{Schubert polynomials and elementary symmetric products}

\author[Oma Makhija]{Oma Makhija}
 
\begin{abstract}
    We study the factorization of Schubert polynomials into elementary symmetric polynomials. We conjecture that this occurs when the permutation corresponding to the Schubert polynomial does not contain the patterns $1432$, $1423$, $4132$, and $3142$. We prove one direction of this and provide progress towards the second direction, including obstructions arising from permutations with a rectangular array of crosses in their bottom pipe dream. This characterization helps us identify new ties between elementary symmetric polynomials and Schubert polynomials. It contributes to the broader understanding of pattern avoidance phenomena in algebraic combinatorics.

\end{abstract}
\maketitle


\section{Introduction}

Schubert polynomials are important polynomials, arising both from geometry and algebra. They have many interesting relationships with other families of polynomials - for example, they contain \textit{Schur polynomials}, \textit{elementary symmetric polynomials}, \textit{homogeneous symmetric polynomials}, and many other important families. Schubert polynomials have particularly interesting relationships with elementary symmetric polynomials; for example, the \textit{Pieri rule} describes a nice combinatorial formula for the product of a Schubert polynomial with an elementary symmetric polynomial. In this writeup, we further explore this relationship between Schubert polynomials and elementary symmetric polynomials by asking when a Schubert polynomial factorizes as a product of any collection of elementary symmetric polynomials.

Many important properties of Schubert polynomials can be characterized in terms of \textit{pattern avoidance}, so we therefore aim for a pattern avoidance-type answer to our question. For example, classically it is known that $\mathfrak{S}_w$ is a single monomial if and only if $w$ avoids $132$; more recently, it has been shown that:
\begin{itemize}
    \item A Schubert polynomial $\mathfrak{S}_w$ is a single standard elementary monomial if and only if $w$ avoids the patterns $312$ and $1432$. \cite{woodruff2025singleSEM}
    \item The principal specialization of the Schubert polynomial at $w$ is bounded below by $1 + p_{132}(w) + p_{1432}(w)$, where $p_u(w)$ denotes the number of occurrences of the pattern $u$ in $w$. \cite{gao2021principalSpecializations}
    \item For any permutation $w$ that avoids a certain set of 13 patterns of length 5 and 6, the Schubert polynomial $\mathfrak{S}_w$ can be expressed as the sum of Schubert polynomials of smaller permutations. \cite{hatam2021determinantalSEM}
\end{itemize}
In this section, we will cover elementary symmetric polynomials, pattern avoidance, Lehmer codes, pipe dreams and Schubert polynomials then state our conjecture. In the second section, we will prove one direction of the conjecture, and in the third section, we will discuss progress towards the second direction of the conjecture.

\subsection{Elementary symmetric polynomials}
For $m\ge 0$ and $k\ge 0$ the \emph{elementary symmetric polynomial} $e_k(x_1,\dots,x_m)$ is
\[
e_k(x_1,\dots,x_m) \;=\; \sum_{1\le i_1<\cdots<i_k\le m} x_{i_1}x_{i_2}\cdots x_{i_k},
\]

Every term in $e_k(x_1 \dots x_m)$ thus has degree $k$, and the multiplicity of every variable in each term is $\leq 1$.

An example of such a polynomial would be $$e_2(x_1,x_2,x_3,x_4) = x_1x_2+x_1x_3+x_1x_4+x_2x_3+x_2x_4+x_3x_4.$$

Note that, while elementary symmetric polynomials are symmetric in the variables $x_1 \dots x_n$, products of elementary symmetric polynomials in different numbers of variables are no longer necessarily symmetric.

In this paper we are interested in when a given (non-symmetric) Schubert polynomial can be written as a product of such elementary symmetric polynomials (possibly in different sets of initial variables).  Concretely, we will ask when a Schubert polynomial $\mathfrak S_w$ equals a product
\[
\mathfrak S_w \;=\; \prod_{t=1}^r e_{b_t}(x_1,\dots,x_{a_t})
\]
for some integers $a_t,b_t\ge1$.
\subsection{Pattern avoidance for permutations}
A permutation $w\in S_n$ \emph{contains} the pattern $\sigma\in S_k$ if there exist indices $1\le i_1<i_2<\cdots<i_k\le n$ so that the relative order of the entries $w_{i_1},\dots,w_{i_k}$ is the same as that of $\sigma_1,\dots,\sigma_k$.  If no such indices exist we say $w$ \emph{avoids} the pattern $\sigma$.

\begin{example}
   $25413$ contains a $231$ pattern - for example, in positions $1, 2$ and $4$. However, it does not contain a $4321$ pattern, since there is no length $4$ decreasing subsequence. 
\end{example}
Permutation patterns originally arose in computer science, in relation to stack-sortable permutations. But they have since spread to many areas of mathematics, and even other areas like chemistry. In the context of Schubert polynomials, many other important properties of Schubert polynomials are characterized by pattern avoidance. For example, $S_w$ is a single monomial if and only if w avoids $132$.
\subsection{Lehmer codes}

The \emph{Lehmer code} of a permutation $w=w_1w_2\cdots w_n$ is the sequence
\[
L(w) = (\ell_1,\dots,\ell_n),\qquad 
\ell_i \;=\; \big|\{\, j>i : w_j<w_i \,\}\big|.
\]
Each $\ell_i$ records how many entries to the right of position $i$ are smaller than $w_i$. The Lehmer code uniquely determines $w$; in fact, taking the Lehmer code of a permutation defines a bijection between permutations in $S_n$ and tuples $(c_{n-1}, c_{n-2} \dots c_0)$ such that $0 \leq c_i \leq i$.

\begin{example}
    The identity has Lehmer code $(0, 0, 0 \dots 0)$. If $w = 23541$, then $L(w) = (1, 1, 2, 1, 0)$. 
\end{example}

The \textit{length} of $w$ (that is, the number of inversions of $w$) is equal to the sum  of the entries of the Lehmer code by definition.

A descent in a permutation (where $w_i > w_{i+1}$) corresponds to a descent in its Lehmer code, meaning $\ell_i > \ell_{i+1}$. We’ll later use simple inequalities like $\ell_{i+1} - \ell_i \le 1$ to express pattern-avoidance conditions in a compact, combinatorial way.
\subsection{Pipe dreams}

\begin{definition}
    A pipe dream for $w$ is a certain filling of an $n \times n$ grid with two kinds of tiles - crossings (\cross) and elbows (\elbowarc) - such that `following a wire' from row $i$ takes you to column $w(i)$. Furthermore, we require that any two wires in the pipe dream cross at most once. 
\end{definition}

(In other words, pipe dreams are certain ways of laying out wiring diagrams/reduced words for $w$ geometrically). 

For the purposes of visual simplicity we will denote the elbows as blank squares and only write in the crosses like : (new figure in new format) should parallel existing figures.

\begin{example}
    The five pipe dreams for $w = 1432$ are depicted in figure 2. The change from the first to the second pipe dream is a non-simple ladder move.     
\end{example}

There is a simple algorithm for generating all pipe dreams for $w$. Namely: the unique left-justified pipe dream for $w$ is called the \textit{bottom pipe dream}. 

\begin{definition}
    The bottom pipe dream for $w$ has $l_i$ crossings in row $i$, followed only by uncrossings, where $l$ is the Lehmer code of $w$. 
\end{definition}

Starting from the bottom pipe dream, we can generate every pipe dream for $w$ via \textit{ladder moves}:

\textbf{Ladder moves.} A simple ladder move is a local transformation that slides a crossing one step up and to the right (along a “ladder”), swapping its position with an adjacent elbow. This preserves the overall wiring and the permutation represented by the diagram. Figure 1 shows  a more complex ladder move which involves a cross ``jumping" above row(s) of crosses in between.

\begin{figure}\label{figure:lader}
    \centering
    \begin{tikzpicture}[line join=round,line cap=round,>=stealth]
    \begin{scope}
    \draw (0,0) grid (2,-5);
    \definecolor{pipeblue}{RGB}{120,180,235}
    \draw[pipeblue, line width=1.6pt] (0,-0.5) to[out=0,in=-90] (0.5,0);
    \draw[pipeblue, line width = 1.6pt] (0,-1.5)--(2,-1.5);
    \draw[pipeblue, line width = 1.6pt] (0,-2.5)--  (2,-2.5);
    \draw[pipeblue, line width = 1.6pt] (0,-3.5)--(2,-3.5);
        \draw[pipeblue, line width = 1.6pt] (0.5,-5)--(0.5,-1)to[out=90,in=180] (1,-0.5)to[out=0,in=-90] (1.5,0);
    \draw[pipeblue, line width = 1.6pt] (0, -4.5) --(1, -4.5) to[out=0,in=-90] (1.5,-4) -- (1.5,-1)to[out=90,in=180] (2,-.5);
    \draw[pipeblue, line width=1.6pt] (1.5,-5) to[out=90,in=180] (2,-4.5);
    \end{scope}
    \begin{scope} [xshift = 3cm]
    \draw (0,0) grid (2,-5);
    \definecolor{pipeblue}{RGB}{120,180,235}
    \draw[pipeblue, line width=1.6pt] (0,-0.5) to[out=0,in=-90] (0.5,0);
    \draw[pipeblue, line width = 1.6pt] (0,-1.5)--(2,-1.5);
    \draw[pipeblue, line width = 1.6pt] (0,-2.5)--(2,-2.5);
    \draw[pipeblue, line width = 1.6pt] (0,-3.5)--(2,-3.5);
    \draw[pipeblue, line width = 1.6pt] (0,-4.5)to[out=0,in=-90](0.5,-4)--(0.5,-1) to[out=90,in=180] (1,-0.5) --(2,-0.5);
    \draw[pipeblue, line width = 1.6pt] (0.5, -5) to[out=90,in=180] (1, -4.5) to[out=0,in=-90] (1.5,-4) -- (1.5,-1)-- (1.5,0);
    \draw[pipeblue, line width=1.6pt] (1.5,-5) to[out=90,in=180] (2,-4.5);
    \end{scope}
    \end{tikzpicture}
    \caption{Ladder Moves}
    \label{fig:placeholder}
\end{figure}
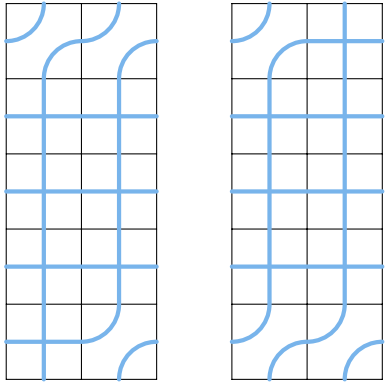

\begin{example}
   In figure 2, the bottom pipe dream is to the left, and the pipe dream on the right is obtained by applying one order $1$ ladder move to the bottom pipe dream.
\end{example}

The bottom pipe dream and the \textit{top pipe dream} will both be important to us. The top pipe dream is the unique top-justified pipe dream for $w$. In Figure 2, the top pipe dream is the bottom left pipe dream.

\begin{figure}\label{figure:1432pipedreams}
    \centering
    \begin{tikzpicture}[line join=round,line cap=round,>=stealth]
    
\begin{scope} [xshift = 0cm]
  \draw (0,0) grid (4,-1);
  \draw (0,-1) grid (3,-2);
  \draw (0,-2) grid (2,-3);
  \draw (0,-3) grid (1,-4);
  \node at (0.5,  0.35) {1};
  \node at (1.5,  0.35) {2};
  \node at (2.5,  0.35) {3};
  \node at (3.5,  0.35) {4};
  \node[left] at (-0.3,-0.5) {1};
  \node[left] at (-0.3,-1.5) {2};
  \node[left] at (-0.3,-2.5) {3};
  \node[left] at (-0.3,-3.5) {4};
  \draw[pipeblue, line width=1.6pt]
    (0,-3.5) to[out=0,in=-90] (0.5,-3)--(0.5, -1)
    to[out=90,in=180] (1,-0.5) to[out=0,in=-90] (1.5, 0);
  \draw[pipeblue, line width=1.6pt]
    (0,-2.5) -- (1,-2.5) to[out=0,in=-90] (1.5,-2)-- (1.5,-1)
    to[out=90,in=180] (2,-.5)  to[out=0,in=-90] (2.5, 0);
  \draw[pipeblue, line width=1.6pt]
    (0,-1.5) -- (2, -1.5) to[out=0,in=-90] (2.5,-1)
    to[out=90,in=180] (3,-0.5)to[out=0,in=-90] (3.5,0);
  \draw[pipeblue, line width=1.6pt] (0,-0.5) to[out=0,in=-90] (0.5,0);
\end{scope}

\begin{scope}[xshift=-5.5cm]
  \draw (0,0) grid (4,-1);
  \draw (0,-1) grid (3,-2);
  \draw (0,-2) grid (2,-3);
  \draw (0,-3) grid (1,-4);
  \node at (0.5,  0.35) {1};
  \node at (1.5,  0.35) {2};
  \node at (2.5,  0.35) {3};
  \node at (3.5,  0.35) {4};
  \node[left] at (-0.3,-0.5) {1};
  \node[left] at (-0.3,-1.5) {2};
  \node[left] at (-0.3,-2.5) {3};
  \node[left] at (-0.3,-3.5) {4};
  \draw[pipeblue, line width=1.6pt]
    (0,-3.5) to[out=0,in=-90] (0.5,-3) to[out=90,in=180](1, -2.5)
    to[out=0,in=-90] (1.5,-2) -- (1.5, 0);
  \draw[pipeblue, line width=1.6pt]
    (0,-2.5) to[out=0,in=-90] (0.5,-2)-- (0.5,-1)
    to[out=90,in=180] (1,-.5) -- (2, -.5) to[out=0,in=-90] (2.5, 0);
  \draw[pipeblue, line width=1.6pt]
    (0,-1.5) -- (2, -1.5) to[out=0,in=-90] (2.5,-1)
    to[out=90,in=180] (3,-0.5)to[out=0,in=-90] (3.5,0);
  \draw[pipeblue, line width=1.6pt] (0,-0.5) to[out=0,in=-90] (0.5,0);
\end{scope}

\begin{scope} [xshift = 5.5cm]
  \draw (0,0) grid (4,-1);
  \draw (0,-1) grid (3,-2);
  \draw (0,-2) grid (2,-3);
  \draw (0,-3) grid (1,-4);
\node at (0.5,  0.35) {1};
\node at (1.5,  0.35) {2};
\node at (2.5,  0.35) {3};
\node at (3.5,  0.35) {4};
\node[left] at (-0.3,-0.5) {1};
\node[left] at (-0.3,-1.5) {2};
\node[left] at (-0.3,-2.5) {3};
\node[left] at (-0.3,-3.5) {4};
\definecolor{pipeblue}{RGB}{120,180,235}
\draw[pipeblue, line width=1.6pt]
  (0,-3.5) to[out=0,in=-90] (0.5,-3)--(0.5, -1) to[out=90,in=180] (1,-0.5) to[out=0,in=-90] (1.5, 0);
\draw[pipeblue, line width=1.6pt]
  (0,-2.5) to (1,-2.5) to[out=0,in=-90] (1.5,-2) to [out=90,in=180] (2,-1.5)  to[out=0,in=-90] (2.5, -1) --(2.5,0) ;
\draw[pipeblue, line width=1.6pt]
  (0,-1.5) -- (1, -1.5) to[out=0,in=-90] (1.5,-1)to[out=90,in=180] (2,-0.5)--(3,-0.5)to[out=0,in=-90] (3.5,0);
\draw[pipeblue, line width=1.6pt] (0,-0.5) to[out=0,in=-90] (0.5,0);
\end{scope}

\begin{scope} [xshift = 1.cm, yshift = -5cm]
  \draw (0,0) grid (4,-1);
  \draw (0,-1) grid (3,-2);
  \draw (0,-2) grid (2,-3);
  \draw (0,-3) grid (1,-4);
\node at (0.5,  0.35) {1};
\node at (1.5,  0.35) {2};
\node at (2.5,  0.35) {3};
\node at (3.5,  0.35) {4};
\node[left] at (-0.3,-0.5) {1};
\node[left] at (-0.3,-1.5) {2};
\node[left] at (-0.3,-2.5) {3};
\node[left] at (-0.3,-3.5) {4};
\draw[pipeblue, line width=1.6pt]
  (0,-3.5) to[out=0,in=-90] (0.5,-3)--(0.5, -2) to[out=90,in=180] (1,-1.5) to[out=0, in=-90] (1.5,-1) -- (1.5, 0);
\draw[pipeblue, line width=1.6pt]
  (0,-2.5) to (1,-2.5) to[out=0,in=-90] (1.5,-2) to [out=90,in=180] (2,-1.5)  to[out=0,in=-90] (2.5, -1) --(2.5,0) ;
\draw[pipeblue, line width=1.6pt]
  (0,-1.5)  to[out=0,in=-90] (0.5,-1) to[out=90,in=180] (1,-0.5) --(2,-0.5) --(3,-0.5)to[out=0,in=-90] (3.5,0);
\draw[pipeblue, line width=1.6pt] (0,-0.5) to[out=0,in=-90] (0.5,0);
\end{scope}

\begin{scope} [yshift = -5cm, xshift = -4cm]
  \draw (0,0) grid (4,-1);
  \draw (0,-1) grid (3,-2);
  \draw (0,-2) grid (2,-3);
  \draw (0,-3) grid (1,-4);
\node at (0.5,  0.35) {1};
\node at (1.5,  0.35) {2};
\node at (2.5,  0.35) {3};
\node at (3.5,  0.35) {4};
\node[left] at (-0.3,-0.5) {1};
\node[left] at (-0.3,-1.5) {2};
\node[left] at (-0.3,-2.5) {3};
\node[left] at (-0.3,-3.5) {4};
\draw[pipeblue, line width=1.6pt]
  (0,-3.5) to[out=0,in=-90] (0.5,-3) to[out=90,in=180](1, -2.5) to[out=0,in=-90] (1.5,-2) -- (1.5, 0);
\draw[pipeblue, line width=1.6pt]
  (0,-2.5) to[out=0,in=-90] (0.5,-2) to[out=90,in=180] (1,-1.5) -- (2,-1.5)  to[out=0,in=-90] (2.5, -1) --(2.5,0) ;
\draw[pipeblue, line width=1.6pt]
  (0,-1.5)  to[out=0,in=-90] (0.5,-1) to[out=90,in=180] (1,-0.5) --(2,-0.5) --(3,-0.5)to[out=0,in=-90] (3.5,0);
\draw[pipeblue, line width=1.6pt] (0,-0.5) to[out=0,in=-90] (0.5,0);
\end{scope}

\end{tikzpicture}
    \caption{Two Pipe Dreams of the permutation 1432}
    \label{fig:placeholder}
\end{figure}
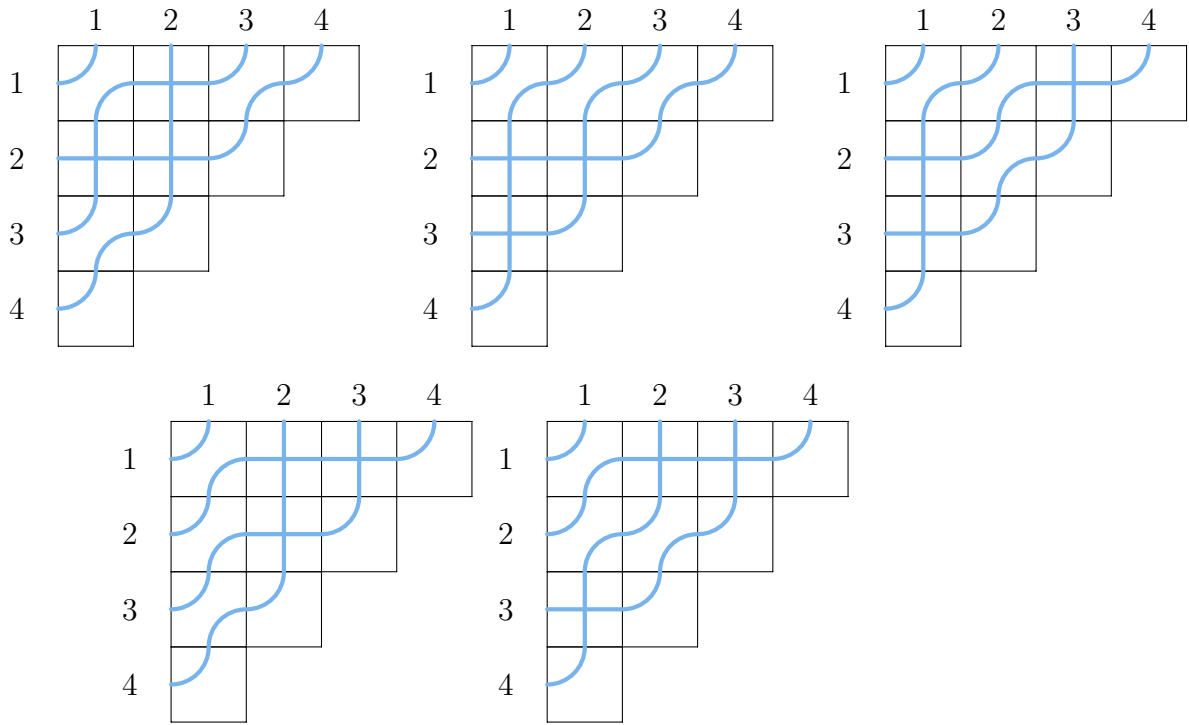
\subsection{Schubert polynomials}

The \textit{weight} of a pipe dream is a monomial, computed as the product over all crossings $C$ of $x_i$, where crossing $C$ is in row $i$. 

\begin{definition}
    The Schubert polynomial $\mathfrak{S}_w$ is defined as 
    \[\mathfrak{S}_w = \sum_{\text{Pipe dreams } P \textit{ for } w} \text{weight}(P)\]
\end{definition}

\begin{example}
    If $w = 1432$, then 
    \[\mathfrak{S}_w = x_2^2x_3 + x_1x_2x_3+x_1^2x_3+x_1^2x_2+x_1x_2^2\]
    Each monomial is contributed by one of the five pipe dreams. 
\end{example}

Notice that $\mathfrak{S}_w$ is always \textit{homogeneous}, since every pipe dream for $w$ contains the same number of crossings, equal to the number of inversions of $w$. The degree of $\mathfrak{S}_w$ is therefore precisely this inversion number. Moreover, the leading monomial in \textit{reverse lexicographic order} corresponds to the Lehmer code of $w$, as it arises from the bottom (left-justified) pipe dream.

\subsection{Why Schubert polynomials?}

From a geometric perspective, Schubert polynomials encode intersection-theoretic information about degeneracy loci, making them indispensable objects in modern Schubert calculus.

From an algebraic perspective, they form a distinguished basis of the polynomial ring, compatible with divided difference operators and the weak Bruhat order, and provide a combinatorial model for the structure constants of the cohomology ring of the flag variety. 
\subsection{Statement of the conjecture}
A Schubert polynomial $\mathfrak{S}_w$ factors as a product of elementary symmetric polynomials if and only if the permutation $w$ avoids the four patterns $1432$, $1423$, $4132$, and $3142$.

That is, for such permutations $w$, there exist integers $a_t, b_t \ge 1$ so that
\[
\mathfrak{S}_w = \prod_{t=1}^r e_{b_t}(x_1, \dots, x_{a_t}).
\]
This conjecture proposes a new pattern-avoidance characterization within Schubert calculus, describing exactly when the rich combinatorics of $\mathfrak{S}_w$ simplifies into a purely symmetric factorization.

\section{Proof of the first direction}

\begin{lemma} The permutation $w$ avoids $1423$ and $1432$ implies its Lehmer code never increases by more than $1$
\end{lemma}

\begin{proof}

Suppose for contradiction that $w$ exists such that $L_{i+1} - L_i \ge 2$ and it avoids $1432$ and $1423$. 
Then there exist at least two distinct indices $p,q>i+1$ with 
$w_p, w_q < w_{i+1}$ but $w_p, w_q > w_i$; otherwise the counts would differ by at most $1$.

Because $w_i < w_{i+1}$ (otherwise $L_{i+1} \le L_i$), 
the relative order of the four entries $w_i, w_{i+1}, w_p, w_q$ forms either the pattern 
$1423$ or $1432$: $w_i$ is the smallest, $w_{i+1}$ the largest, 
and $w_p, w_q$ are the two middle values ordered increasingly or decreasingly. 
Thus $w$ contains one of the forbidden patterns $1423$ or $1432$.

This contradicts the avoidance assumption. 
Hence no such pair $p,q$ can exist, and we must have $L_{i+1} - L_i \le 1$ for every~$i$.
\end{proof}

\begin{definition}[Column in a pipe dream]
A \emph{column} of a pipe dream is a maximal contiguous set of crossings in a single column such that:
\begin{itemize}
    \item The crossings are consecutive (no elbows in between),
    \item The set is as tall as possible without being interrupted by an elbow or the edge of the grid.
\end{itemize}
Equivalently, a column can be thought of as a vertical "chunk" of crossings in the pipe dream that will contribute together to a factor in a product of elementary symmetric polynomials.
\end{definition}

\begin{lemma}
If the permutation $w$ avoids the patterns $3142$ and $4132$ that implies that in every pipe dream for $w$, the following holds:  

\noindent\textbf{Diagonal separation property.} For any two columns $c<c'$ of the pipe dream, the northeast-going diagonal (slope $+1$) through the \emph{top} cross of the lower-numbered column $c$ lies strictly below (by at least one lattice unit) the northeast-going diagonal through the \emph{bottom} cross of the higher-numbered column $c'$.
\end{lemma}

\begin{proof} 
Let $w$ be a permutation avoiding $3142$ and $4132$, and let $P$ be the bottom pipe dream of $w$.
Assume, for the sake of contradiction, that the diagonal separation property does not hold.
Then there exist columns $c < c'$ such that the northeast-going diagonal through the top cross of column $c$ is not strictly below the northeast-going diagonal through the bottom cross of column $c'$.

Denote the row containing the bottom cross of column $c'$ as $i$ such that it corresponds to  $w_i$ in the permutation. Denote the row containing the top cross of column $c$ as $j$ such that it corresponds to  $w_j$ in the permutation. 

In order for these to be separate columns we also know that there must exist some $k$ such that $i<k<j$. 

There are a nonzero amount of crosses in row $j$. Eventually after row $j$ there will be a row with less crosses (by the definition of a pipe dream as the bottom row must have zero crosses). Let this row be row $l$. Since the vertical segment of the cross(es) in row $j$ need to have a pipe, the pipe that starts at row $l$ will go through the last cross in row $j$. Therefore the (whats the word for it) of the permutation that row $l$ corresponds (denote as $w_l$) must be less than $w_i$ and $w_j$. It will however be greater than $w_k$ as it is in a later row and (more thorough explanation here).

Therefore we have that there exist $w_i>w_k>w_l$ and $w_j>w_k>w_l$ where $i<k<j<l$. This implies that either a $4132$ or $3142$ pattern exists.

\end{proof}

\begin{lemma} The Schubert polynomial $S_w$ can be expressed as the product of elementary symmetric polynomials if \begin{enumerate}
    \item the permutation satifies the diagonal seperation property and
    \item its Lehmer code never increases by more than $1$
\end{enumerate}
\end{lemma}

\begin{proof}
Begin with the bottom pipe dream and group the crosses into columns per defintion 2.2. 

From the diagonal seperation property and the Lehmer code never increasing by more than $1$, we have that the columns never interact with each other as they slide. As singular columns do not interact with each other and the Schubert polynomials resulting from pipe dreams of columns are elementary symmetric polynomials, the resulting Schubert polynomial will be the product of each of these Schubert polynomials.

\end{proof}

So, we have proved one direction of the conjecture by combining Lemmas 2.1, 2.3 and 2.4. 

\begin{figure}
    \centering
    \begin{tikzpicture}
    \draw (0,0) rectangle (6,6);
    \foreach \x in {1,...,5} \draw ( \x ,0) -- ( \x ,6);
    \foreach \y in {1,...,5} \draw (0, \y ) -- (6, \y);
    \node[text=red] at (0.5,5.5) {+};
    \node[text=red] at (1.5,5.5) {+};
    \node[text=red] at (2.5,5.5) {+};
    \node[text=red] at (0.5,4.5) {+};
    \node[text=red] at (1.5,4.5) {+};
    \node[text=blue] at (0.5,1.5) {+};
    \end{tikzpicture}
    \caption{Pipe Dream of the permutation 431265}
    \label{fig:placeholder}
\end{figure}
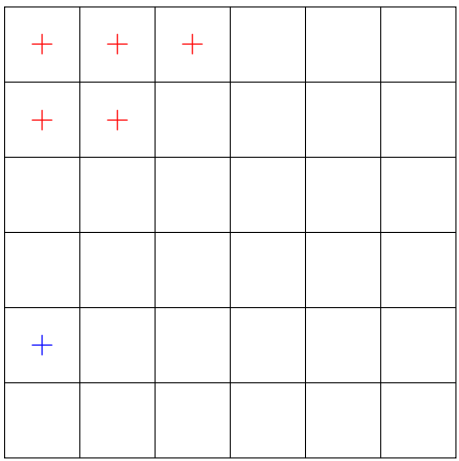

\section{Progress towards the second direction}

\begin{definition}
The \textbf{bottom pipe dream} of $w$ is the unique left-justified pipe dream for $w$ where the number of crossings in row $i$ is $L(i)$. (Here, left-justified means that every row is a string of crossings followed by a string of elbows).
\end{definition}
\begin{definition}
The \textbf{top pipe dream} of $w$ is the unique top-justified pipe dream for $w$ where the number of crossings in row $i$ is $L(i)$. (Here, top-justified means that every column read top to bottom is a string of crossings followed by a string of elbows). 
\end{definition}

\begin{lemma} 
The Schubert polynomial of a permutation $w$ whose Lehmer code corresponds to a rectangular block
\[
L(w) = ( \underbrace{0,\dots,0}_{A}, \underbrace{N,\dots,N}_{B}, \underbrace{0,\dots,0}_{\ast} )
\]
cannot be written as a product of elementary symmetric polynomials.
\end{lemma}
\begin{proof} 
We argue by contradiction. Suppose
\[
\mathfrak S_{w} \;=\; f_1 f_2 \cdots f_r
\]
with each $f_i$ an elementary symmetric polynomial.

In the top pipe dream, the highest monomial is
$M_{top} = x_1^Nx_2^N\dots x_{B}^N$ . In the bottom pipe dream, the resulting monomial is $M_{bot} = x_{A}^N x_{A+1}^N \cdots x_{A+B}^N$ . 

Since each elementary symmetric polynomial can contribute at most one $x_1$ to a resulting monomial in the product in order for the largest term to be $x_1^N \,x_{B}^N$ there must be exactly $N$ terms. 

From the above, each factor $f_i$ must simultaneously:  contain the monomial $x_1x_2\cdots x_B$ (to allow $M_{\mathrm{top}}$) and  contain the monomial $x_Ax_{A+1}\cdots x_{A+B}$ (to allow $M_{\mathrm{bot}}$). 

But if a factor has degree $<B$, then its monomials cannot contain all of $x_1, x_2, \dots,x_B$ simultaneously. That would make it impossible to form $M_{top}$, which needs all of them present in every factor-choice. So every factor must have degree at least $B$. The total degree of this Schubert polynomial must be $B\cdot N$ (the total number of boxes) as this is a constant quantity. Since no polynomial in the product can have a degree of less, they all must have a degree of $B$.
Thus this forces:
$r = N \quad \text{and} \quad \deg(f_i) = B \;\; \text{for all $i$}.$

Given Step~2, each factor must be some $e_B(\,\cdot\,)$. 
But the bottom monomial requires that in its largest term we can get 
\[
x_A x_{A+1} \cdots x_{A+B},
\]
while the top requires 
\[
x_1 x_2 \cdots x_B.
\]
The only way both can happen is if each factor is
\[
f_i \;=\; e_B(x_1, \dots, x_{A+B}).
\]
This will result in $B\cdot N$ monomials. 

However since the columns are not independent in a rectangle there will be strictly less monomials in the product than $B\cdot N$. Therefore we have a contradiction and the Schubert polynomial $\mathfrak S_w$ cannot be expressed as the product of elementary symmetric polynomials.

\end{proof}
\bigskip

\nocite{*}
\bibliographystyle{amsplain}   
\bibliography{references}

\end{document}